\documentclass{article}
\usepackage{amssymb,amsfonts,amsmath,amsthm,amscd}
\usepackage{makeidx}
\usepackage[english]{babel}
\makeindex

\theoremstyle{thm}
\newtheorem{thm}{Theorem}[section]

\theoremstyle{lem}
\newtheorem{lem}[thm]{Lemma}
\newtheorem{prop}[thm]{Proposition}
\newtheorem{defn}[thm]{Definition}
\theoremstyle{rem}
\newtheorem{rem}[thm]{Remark}
\newtheorem{exe}[thm]{Example}
\newcommand{\I}{\mathcal{I}}
\newcommand{\C}{\mathcal{C}}
\newcommand{\D}{\mathcal{D}}
\newcommand{\G}{\mathcal{G}}
\newcommand{\J}{\mathcal{J}}
\newcommand{\Ll}{\mathcal{L}}
\newcommand{\R}{\mathcal{R}}
\newcommand{\B}{\mathcal{B}}

\newcommand{\Z}{\mathbb{Z}}
\newcommand{\af}{\alpha}
\newcommand{\bt}{\beta}
\newcommand{\f}{\textbf{f}}

\title{\textbf{Partial Groupoid Actions on $R$-Categories: Globalization and the smash product}}
\author{V\'{\i}ctor Mar\'{\i}n $^{\text{1}}$\\
   \small $^{\text{1}}$Departamento de Matem\'{a}ticas y Estad\'{i}stica \\
   \small Universidad del Tolima\\
   \small Santa Helena, Ibagu\'{e}, Colombia\\
   \small e-mail:vemarinc@ut.edu.co\\\\
   H\'{e}ctor Pinedo $^{\text{2}}$\\
   \small $^{\text{2}}$ Escuela de Matem\'{a}ticas\\
   \small Universidad Industrial de Santander\\
   \small Cra. 27 calle 9, Bucaramanga, Colombia\\
   \small  e-mail:hpinedot@uis.edu.co}
\date{\today}
\begin{document}
\maketitle
\begin{abstract}
\noindent
In this article, we introduce the concept of  partial groupoid  actions  on $R$-semicategories as well as we give criteria for existence of a globalization of it. This point of view is a generalization of the notions of partial  groupoid actions on  rings and partial group action  on an  $R$-semicategory. We also define the notions of partial skew groupoid 
category, smash product and  describe functorial relations
between them, in particular we show that the smash product is a Galois covering of its associated skew groupoid 
category.
\end{abstract}

\noindent
\textbf{2010 AMS Subject Classification:} Primary 20L05, 18E05. Secondary 16W55, 20N02.\\
\noindent
\textbf{Key Words:} Partial groupoid action,  $R$-semicategory,  globalization, smash product.

\section{Introduction} An algebraic  study of  group actions  on categories  was presented  in \cite{Go}, while
for a commutative ring $R,$ the notion of a partial group  action on a   $R$-semicategory was introduced and studied  in \cite{CFM}, for the readers convenience we recall that the  concept of a semicategory  or non-unital category is like that of category but omitting the requirement of identity-morphisms. By an $R$-semicategory we mean a semicategory $\C$ such that the morphism set $_yC_x$ from
an object $x$ to an object $y$ is an $R$-module and the composition  $ {_z\C_y}\times {_y\C_x}\ni (f,g)\mapsto f g\in  {_z\C_x}$ is $R$-bilinear, for each $x,y,z \in \C_0.$ In topology, an example of  semicategory can be formed from the category of metric spaces and short maps, by taking  the nonempty spaces and strictly contractive functions.

  The definition of  semifunctor between two semicategories is similar to the
definition of a functor between categories, where we only  drop  assumptions related to
the unit morphism.
\smallskip

 On the other hand \textit{Groupoids} are usually presented as small categories whose morphisms are
invertible. They are natural extension of groups,   
  we let ${\rm mor}({\G})$   to be the set of morphisms of $\G.$   For a groupoid $\G$ and $g\in {\rm mor}({\G})$, the morphisms
$d(g):=g^{-1}g$ and $r(g):=gg^{-1}$ are called the {\em domain
identity} and {\em range identity} of $g$, respectively. An element
$e\in {\rm mor}({\G})$ is  an {\em identity} of $\G$ if $e=d(g)=r(g^{-1}),$
for some $g\in {\rm mor}({\G})$. The set of identities of $\G$ is  denoted by
$\G_0$.  
 Recall that given 
$g,h\in {\rm mor}({\G})$, the element $gh$ exists if and only if $d(g)=r(h)$. In
this case, $d(gh)=d(h)$ and $r(gh)=r(g)$. We denote by $\G^2$  the
subset of pairs $(g,h)\in \G\times \G$ such that $gh$ exists, and for $e\in \G_0$ we let $\G(-, e)$ to be the set of $g\in {\rm mor}({\G})$  such that $r(g)=e,$  analogously  one defines $\G(e,-).$ The set $\G_e:=\G(-, e)\cap \G(e,-)$  is the so called  {\em
principal group associated to e}. 
  For more details about groupoids, the interested reader may consult  e.~g. \cite{L}. 

Partial actions of groupoids have been a subject of increasingly study and have been considered in several branches,  for instance in \cite{G} the author construct a Birget-Rhodes expansion $\G^{\rm{BR}}$ associated to a ordered groupoid $\G$ and shows that  it  classifies partial actions of $\G$ on sets, in the topological context  they were treated in \cite{ME, ME1, NY}, where the globalization problem was considered.  On the other hand,  ring theoretic results of  global and partial actions of groupoids on algebras are obtained in 
 \cite{B, BFP, BP, GY, NYOP, PF}, in \cite{BPP} the authors study  the existence of connections between partial groupoid actions and partial group actions,  while  Galois theoretic  results for groupoid actions were obtained in \cite{CT, PT1, PT}.   

In this work  we introduce the concept of a partial groupoid action on an $R$-semicategory,  obtaining  a common  ge\-ne\-ra\-li\-za\-tion  of \cite[Definition 3.2]{CFM}   of partial group actions on $R$-semicategories  and  the concept of partial groupoid action on $R$-algebras \cite{BP}, which can be considered as $R$-semicategories with a single  object.

This work is divided as follows. After the introduction, 
in section 2 we present some  notions  and facts which we will use throughout the work. 
In section 3, we introduce the definition of  globalization for a partial action of groupoid on semicategories and give in Theorem \ref{teogloba} necessary and sufficient conditions for the existence of such a globalization, generalizing  similar results of \cite{CFM} and \cite{BP} (see Remark \ref{cfdp}).
 In section 4,  we associate to a partial action $\alpha$ of  $\G$ on a $R$-semicategory $\C$ a non-necessarilly associative $R$-semicategory $\C* _\alpha \G$, which we call the partial skew groupoid semicategory, and is analogous to the skew group semicategory defined in \cite{CFM},  a condition for the associativity of morphisms in $\C* _\alpha \G$ is presented in  Proposition \ref{assoc} and a Morita context between algebras associated to $\C* _\alpha \G$ and the skew groupoid semicategory induced by the globalization of $\alpha$ is given in Theorem \ref{morequiv}. Finally, in section 5 we define the quotient semicategory $\C/\G$, and show that it makes sense when $\C$ is a free category,   as in the case of group actions, one says that  $\C$ is a Galois
covering of the quotient $\C/\G.$ 
Furthermore,  we define smash product semicategory $\B\#\G$ and in  Lemma \ref{cat} we give necessary conditions for it to be a category ,  the principal result of this section is  Theorem \ref{Galois} which shows that there exist a global action $\alpha$ on $\B\#\G$ such that $\B\#\G$ is free $\G$-category and a Galois covering of the $\G$-graded  semicategory $\B\otimes \G$, with objects $(\B\otimes \G)_0=\B_0\times \G_0$ and  morphism $_{(y,f)}(\B\otimes \G)_{(x,e)}=\bigoplus_{g\in _f\mathcal{G}_e}{_y\mathcal{B}_x}^g.$
\medskip

Throughout the work $R$ will denote commutative ring with identity element, and will work only with small $R$-semicategories, that is $R$-semicategories $\C$ in which their class of objects $\C_0$ is a set. 
\section{Partial actions on $R$-semicategories}
We next establish our basic  definitions and results.

Partial actions of categories on sets and topological spaces were defined in \cite[Definition 7]{NY}, while  partial actions of groupoids on rings were introduced in \cite[p. 3660]{BP}.  For the reader's convenience we  recall the definition of a partial action of a groupoid on a set and a ring.

Following \cite{L},  a partial function set $\phi\colon X\to Y$ is a map $\phi: A\to B$, where $A$ and $B$ are subsets of $X$ and $Y$ respectively. Now we recall from \cite{NY} the next.

\begin{defn} \label{NYG}Let $\G$ be a groupoid and $X$ a set. A partial action of $\G$ on $X$ is a partial function ${\rm mor}(\G)\times X\to X$ denoted by $(g,x)\to g\cdot x,$ if $g\cdot x,$ is defined 
such that
\begin{enumerate}
\item[(PA1)] If $ g \cdot x$ is defined, then  $g^{-1}\cdot (g\cdot x)$ is defined, and $g^{-1}\cdot (g\cdot x)=x$.
\item[(PA2)] If $ g\cdot (h\cdot x)$ is defined, then $(gh)\cdot x$  is defined, and $g\cdot (h\cdot x)=(gh)\cdot x,$ for all $(g,h)\in \G^{2}.$
\item[(PA3)] For every $x \in X,$ there is $e \in \G_0$ such that $e \cdot x$ is defined. If $f \in \G_0$ and  $x \in X,$ are
chosen so that $f \cdot x$ is defined, then $f \cdot x=x$.
\end{enumerate}

\end{defn}

By \cite[Remark 28]{NY} partial groupoid actions on sets can be equivalently formulated in terms of partial defined maps  as follows.

\begin{defn}\label{ny}
A  \emph{partial action} $\alpha$ of a groupoid $\mathcal{G}$ on set $X$ is a pair $\alpha=(D_g, \alpha_g)_{g\in {\rm mor}(\G)}$ where for each $g\in {\rm mor}(\G), D_g\subseteq  D_{r(g)}\subseteq X$  and $\alpha_g:D_{g^{-1}}\rightarrow D_g$ are bijections such that:
\begin{enumerate}
\item[(i)] $X=\bigcup_{e\in \G_0}D_e$ and  $\alpha_e$ is the identity map ${\rm id}_{D_e}$ of $D_e,$ for all $e\in \mathcal{G}_0$;
\item[(ii)] $\alpha_{g}(D_{g^{-1}}\cap D_{h})=D_g\cap  D_{gh}$;
\item[(iii)] $\alpha_g(\alpha_h(x))=\alpha_{gh}(x)$, for every $x\in \alpha_h^{-1}(D_{g^{-1}}\cap D_h)$,
\end{enumerate}
 for each $(g,h)\in \mathcal{G}^2$.
\end{defn}

\begin{exe}
Consider $X=\{e_1,e_2,e_3\}$ and let $\mathcal{G}=\{d(g),r(g),g,g^{-1}\}$ be a groupoid. Let us take the subsets $D_{d(g)}=\{e_1, e_2\}, D_{r(g)}=D_g=\{e_3\},$ and $ D_{g^{-1}}=\{e_1\}$ of $X$ and define $\alpha$ by $\alpha_{d(g)}={\rm id}_{D_{d(g)}}, \alpha_{r(g)}={\rm id}_{D_{r(g)}}, \alpha_g(e_1)=e_3, \,\alpha_{g^{-1}}(e_3)=e_1$. It is easy to see that $\alpha$ is a partial action of $\mathcal{G}$ on $X$.
\end{exe}

\begin{rem} In \cite[Definition 2.4]{GY} the authors also present  the notion of a partial action of  groupoid $\G$ on a set $X$, the only difference  with Definition \ref{ny} is that the condition $X=\bigcup_{e\in \G_0}D_e$ is not required. We preffer Definition \ref{ny} because  by adding this  requirement we get the  advantage
that  partial groupoid actions,
in the case when $\G$ is a group,  coincides with the classical definition of partial group actions on sets (see \cite[Definition 1.2]{EX}).
\end{rem}
The concept of partial action of a groupoid on a ring is similar to Definition \ref{ny}. Indeed, we have the following.

\begin{defn}\label{BP} A  \emph{partial action} $\alpha$ of a groupoid $\mathcal{G}$ on a ring $A$ is a pair $\alpha=(D_g, \alpha_g )_{g\in {\rm mor}(\G)}$ where for each $g\in {\rm mor}(\G),$ one has that  $D_{r(g)}$  is an ideal of $A,$ $D_g$ is an ideal of  $D_{r(g)},$   and $\alpha_g:D_{g^{-1}}\rightarrow D_g$ are ring isomorphisms such that:

\begin{enumerate}
\item[(i)]  $\alpha_e$ is the identity map ${\rm id}_{D_e}$ of $D_e,$ for all $e\in \mathcal{G}_0$;
\item[(ii)] $\alpha_h^{-1}(D_{g^{-1}}\cap D_h)\subseteq   D_{(gh)^{-1}}$;
\item[(iii)] $\alpha_g(\alpha_h(x))=\alpha_{gh}(x)$, for every $x\in \alpha_h^{-1}(D_{g^{-1}}\cap D_h)$,
\end{enumerate}
 afor each $(g,h)\in \mathcal{G}^2$.
\end{defn}

\begin{defn}
Let $\mathcal{G}$ be a groupoid and $X$ a set or a ring. A partial action $\af$ of $\mathcal{G}$  on $X$ is  global if $\alpha_g\circ \alpha_h=\alpha_{gh},$ for all $(g,h)\in \mathcal{G}^2$.
\end{defn}

 For a partial action of $\G$ on an object $X$ and   $x\in X$ we denote $\G^x=\{g\in {\rm mor}( \mathcal{G})\mid x\in D_{g^{-1}}\}$    and $\G \cdot x=\{gx\mid g\in\G^x\},$ the $\G$-orbit of $x.$ Notice that by (PA3) the set $\G_0\cap \G^x$ is always non-empty.
\smallskip

 Every groupoid acts globally on itself by multiplication. Indeed, we have the following.
\begin{exe}\label{exe3} Let $\G$ be a groupoid and  $g\in {\rm mor}(\G),$ set $\G_g=\G(-, r(g))$ and 
$$\beta_g\colon\G(-, d(g))\ni h\mapsto gh \in\G(-, r(g)).$$
 It is not difficult to show that the family $\beta=(\beta_g, \G_g)_{g\in {\rm mor}(\G)}$ is a global action of $\G$ on itself. Moreover, for any $s\in {\rm mor}(\G)$ one has that $\G\cdot s=\G(d(s),-).$ Indeed it is clear that $\G\cdot s\subseteq \G(d(s),-),$ conversely if $u\in \G(d(s),-),$ then $(u,s^{-1})\in \G^2$ and   $u=(us^{-1})s\in \G\cdot s,$ as desired.
\end{exe}
 Partial action of groups on  $R$-semicategories were introduced in \cite{CFM}. Now we extend this notion to the concept of partial groupoid actions, but first we recall the following. 
\begin{defn}\label{idmor}\cite [Definition 2.2. and Definition 2.5] {CFM} Let $\C$ be an 
 $R$-semicategory, and $\I$   a collection  of morphisms in $\C$. Then
\begin{itemize} 
\item We say that $\I$ is an ideal of $\C$ if  for $f\in \I,$ and $g,h$ morphisms in $\C,$ one has that  $gf$ and $fh$ are in $\I$ whenever $gf$ and $fh$ are defined, and   ${_y\I_x}$ is an $R$-submodule of  ${_y\C_x}$, where ${_y\I_x}={_y\C_x}\cap \I.$ 
\item  A morphism e in ${_x\I_x}$ is called a left (respectively right) local identity if, $ef=f$ for all $f\in {_x\I_y}$, and (respectively $fe=f$ for all $f\in {_y\I_x}$). A local identity is a left and right local identity.
\end{itemize}
\end{defn}

We write $I\unlhd \C$ to denote that $\I$ is an ideal of $\C.$ 

\begin{defn}\label{apgc}
Let $\mathcal{G}$ be a groupoid, $\C$ an  $R$-semicategory. We say that $\alpha=(\mathcal{I}^g,  \alpha^g)_{g\in {\rm mor}(\mathcal{G})}$ is a partial action of $\mathcal{G}$ on $\C$ if the following conditions hold:
\begin{enumerate}
\item[(i)] $\mathcal{G}$ acts partially on the set objects $\C_0$ of $\C.$ This partial action will be denoted by $\alpha_0=( \C_0^g , \alpha_0^g )_{g\in {\rm mor}{\G}}$;
\item[(ii)] For each $g\in \mathcal{G}$ there exists a subset $\mathcal{I}^g$ of morphisms in $\C$ such that ${_a\mathcal{I}^g_b}=0$ if $\{a, b\}$ is not a subset of $\C_0^g$;
\item[(iii)] There are equivalence of $R$-semicategories $\alpha^g:\mathcal{I}^{g^{-1}}\rightarrow \mathcal{I}^g$, where $\mathcal{I}^g\unlhd\mathcal{I}^{r(g)}\unlhd \C, $ for each $g\in \mathcal{G}$, such that for $f\in \,_y\mathcal{I}^{g^{-1}}_x$ and $\{x,y\}\subseteq \C_0^{g^{-1}}$, one gets $\alpha^g(f)\in {_{gy}{\mathcal{I}^g}_{gx}};$
\item[(iv)] $\alpha^e$ is the identity map  of $\mathcal{I}^e$;
\item[(v)] For  objects $x,y\in \C_0^h\cap \C_0^{g^{-1}}$,
           $\alpha^{h^{-1}}({_y{\mathcal{I}^h}_x} \cap {_y{\mathcal{I}^{g^{-1}}_x}})\subseteq {_{h^{-1}y}\mathcal{I}^{(gh)^{-1}}_{h^{-1}x}},$
\item[(vi)] If $x,y\in \C_0^h\cap \C_0^{g^{-1}}$ and $f\in \alpha^{h^{-1}}({_y{\mathcal{I}^h}_x} \cap {_y\mathcal{I}^{g^{-1}}_x})$, then $\alpha^g(\alpha^h(f))=\alpha^{gh}(f)$,
\end{enumerate}
for all $e\in \mathcal{G}_0$ and $(g,h)\in \mathcal{G}^2$.
\end{defn}

\begin{rem} $\alpha=( \mathcal{I}^g,  \alpha^g )_{g\in {\rm mor}(\mathcal{G})}$ is a partial action of $\mathcal{G}$ on $\C.$ Then:
\begin{itemize}
\item  The family of ideals $\{\mathcal{I}^g\}_{g\in {\rm mor}(\mathcal{G})}$ satisfy  $\I^g_0=\C_0,$ for each morphism $g$ of $\G.$
\item If we require that   $\alpha_0$ is global, then the pair  $\alpha^{(e)}=(\mathcal{I}^g, \alpha^g)_{g\in \G_e}$ is a partial action (in the sense of \cite[Definition 3.2]{CFM}) of  $\G_e$ on the $R$-semicategory, $\mathcal{I}^e$, for all $e\in \mathcal{G}_0.$  
\end{itemize}
\end{rem}
\begin{defn}
Let $\alpha=( \mathcal{I}^g,  \alpha^g)_{g\in \mathcal{G}}$ be a partial action of a groupoid $\mathcal{G}$ on $\C$. We say that $\alpha$ is global if $\alpha^g\alpha^h=\alpha^{gh}\,\,{\rm and}\,\, \alpha_0^g\alpha_0^h=\alpha_0^{gh}\, \text{for all}\, (g,h)\in \mathcal{G}^2.$
\end{defn}
We have the following.

\begin{lem}\label{global}
Let $\alpha=( \mathcal{I}^g,  \alpha^g)_{g\in \mathcal{G}}$ be a partial action of a groupoid $\mathcal{G}$ on an $R$-semicategory $\C$. Then, the following statements hold:
\begin{enumerate}
\item[(i)] $\alpha$ is global if and only if $\I^g=\I^{r(g)}$  and $\C_0^g =\C_0^{r(g)} $  for each $g\in {\rm mor}(\G),$
\item[(ii)] $\alpha^{-1}_g=\alpha_{g^{-1}},$    for each $g\in {\rm mor}(\G)$;
\item[(iii)] $\alpha^g({_y\mathcal{I}^{g^{-1}}_x}\cap {_y\mathcal{I}^h_x})={_y\mathcal{I}^g_x}\cap {_y\mathcal{I}^{gh}_x}, $ for any  $(g,h)\in \mathcal{G}^{2}$.
\end{enumerate}
\end{lem}

\begin{proof}
Similar to \cite[Lemma 1.1]{BP}.
\end{proof}

\begin{exe}\label{exe1}
We consider 
the  R-semicategory $\C$ with
\begin{enumerate}
\item $\C_0=\{x, y\}.$
\item Given $ u,v \in \C_0$ let ${_u\C_v}=Re_1\oplus Re_2\oplus Re_3,$  where $e_1,e_2,e_3$ are pairwise orthogonal central idempotents with sum 1.
\item For all $u,v,w \in \C_0$ an $R$-bilinear map $\cdot: {_u\C_v}\times {_v\C_w}\to {_u\C_w};$ given by multiplication.
\end{enumerate}
Take the groupoid $\G=\{d(g),r(g),g,g^{-1}\}$;   then 
 $\G$ acts partially on $\C_0$ via $\af_0,$ where: 
$$\C_0^{r(g)}=\C_0^{d(g)}=\C_0,\,\,\, \C_0^g=\{x\}, \,\,\,\text{and}\,\,\, \C_0^{g^{-1}}=\{y\}$$ 
and $\alpha_0^g:\C_0^{g^{-1}}\to \C_0^g; \, y \mapsto x, \, \alpha_0^{g^{-1}}:\C_0^g\to \C_0^{g^{-1}};\, x\mapsto y,$  $\alpha_0^{r(g)}={\rm id}_{\C_0^{r(g)}}$ and $\alpha_0^{d(g)}={\rm id}_{\C_0^{d(g)}}$.\\
Now consider the ideals of $\C$  given by 
\begin{itemize}
\item $_u{\mathcal I}^g_v=Re_3,$ if $(u,v)=(x,x)$ and  $_u{\mathcal I}^g_v=0,$ if $(u,v)\neq (x,x). $
\item $_u{\mathcal I}^g_v=Re_3,$ for all $u,v\in \C_0.$
\item $_u\mathcal{I}^{g^{-1}}_u=Re_1$ if $(u,v)=(y,y)$ and $_u\mathcal{I}^{g^{-1}}_u=0,$ if $(u,v)\neq (y,y). $
\item $_u\mathcal{I}^{d(g)}_v=Re_1\oplus Re_2,$ and $_u{\mathcal I}^{r(g)}_v=Re_3,$ for all $u,v\in \C_0.$
\end{itemize}

Then $\alpha=(\mathcal{I}^g, \alpha^g)_{g\in {\rm mor}(G)}$ is a partial action of $\G$ in $\C,$ where   $\alpha^g(ae_1)=ae_3, \, \alpha^{g^{-1}}(ae_3)=ae_1, \, \alpha^{d(g)}=id_{\mathcal{I}^{d(g)}}, \, \alpha^{r(g)}=id_{\mathcal{I}^{r(g)}},$ for each $a\in R$.

\end{exe}
We obtain a partial action of a groupoid by restriction of a (global) groupoid action, in a standard way.

\begin{exe}\label{indu}(\textbf{Induced partial action})
 Let $\C$ be an  R-semicategory and $\beta=( E^g, \beta^g)_{g\in {\rm mor}(\G)}$ a global action  of a groupoid $\mathcal{G}$ on $\C.$ In particular,  there is a global action  $\beta_0=(\C_0^g, \beta_0^g)_{g\in {\rm mor}(\G)}$    on $\C_0.$ Let $\I$ be an ideal of $\C,$  take $e\in \G_0$ and  set $\I_0^e=\I_0\cap \C_0^e$. The partial action $\alpha_0=(\I_0^g, \alpha_0^g)_{g\in {\rm mor}{\G}}$ on $\I_0$ is defined as restriction of $\beta$, that is, $$\I_0^g=\I_0^{r(g)}\cap \beta_g(\I_0^{d(g)})\,\,\text{ and }\,\,\alpha_0^g={\beta_0^g}\restriction  \I_0^{g^{-1}},$$  for all $g\in{\rm mor}(\G)$. 

Now, for $g\in {\rm mor}(\G), x,y\in \G_0$  define  $\mathcal{I}^g$ by:
\begin{itemize}
\item If  $\{x, y\}$ is not a subset of $\I_0^g,$ set $_x\mathcal{I}^g_y=0$ ;
\item If  $\{x, y\}\subseteq \I_0^g,$ then $\{x, y\}\subseteq \I_0^{r(g)}\subseteq  \C_0^{r(g)}=\C_0^{g},$ so $\beta_{g^{-1}}(x)=g^{-1}x$ and $\beta_{g^{-1}}(y)=g^{-1}y$ are well defined. Thus we set 
$${_y\mathcal{I}^g_x}=({_y\I_x}\cap {_yE^g_x})\cap \beta_g({_{g^{-1}y}\I_{g^{-1}x}}\cap {_{g^{-1}y}E^{g^{-1}}_{g^{-1}x}}),$$

In particular, ${_y\mathcal{I}^{r(g)}_x}={_y\I_x}\cap {_yE^{r(g)}_x}.$
\end{itemize} Finally set $\alpha^g=\beta^g\mid_{\I^{g^{-1}}}.$ Then  $\alpha=(\mathcal{I}^g, \alpha^g)_{g\in \mathcal{G}}$ is a partial action of $\mathcal{G}$ on $\I.$  Indeed, by construction we  have (i), (ii), (iv) and (vi) in Definition \ref{apgc}. To check (iii) let us show first that $\mathcal{I}^g\unlhd \I^{r(g)}$ and  $\I^{r(g)}\unlhd \I,$ for all $g\in {\rm mor}(\G).$ For this, take  $x,y\in \I_0$  we only consider the case when  $\{x, y\}\subseteq  \I_0^g.$  If $f\in {_y\mathcal{I}^g_x}$ and $l,m$ are morphisms such that $l\in {_z\I^{r(g)}_y}$ and $m\in {_x\I^{r(g)}_u},$ with $u,z \in \I_0$.  We need to show that
$fm\in  {_y\mathcal{I}^g_u}$ and $lf\in  {_z\mathcal{I}^g_x}.$ To prove the first assertion, notice that the fact that $E^g\unlhd E^{r(g)}$ implies  $fm\in {_y\I_u}\cap {_yE^g_u}.$ Moreover,  since $\beta$ is global there are 
$$\widetilde{f}\in {_{g^{-1}y}\I_{g^{-1}x}}\cap {_{g^{-1}y}E^{g^{-1}}_{g^{-1}x}}\,\,\,\text{  and}\,\,\,\, \widetilde{m}\in {_{g^{-1}x}E^{g^{-1}}_{g^{-1}u}}$$ 
such that $f=\beta_g(\widetilde{f}),$ and $m=\beta_g(\widetilde{m})$ respectively, so $\widetilde{f}\widetilde{m}\in  {_{g^{-1}y}\I_{g^{-1}u}}\cap {_{g^{-1}y}E^{g^{-1}}_{g^{-1}u}},$ thus 
$$fm=\beta_g(\widetilde{f})\beta_g(\widetilde{m})=\beta_g(\widetilde{f}\widetilde{m})\in ({_z\I_u}\cap {_zE^g_u})\cap \beta_g({_{g^{-1}z}D_{g^{-1}u}}\cap {_{g^{-1}z}E^{g^{-1}}_{g^{-1}u}})={_z\mathcal{I}^g_u}.$$
 In an analogous way one  shows that $lf\in   {_z\mathcal{I}^g_x},$ and using tha fact that $ E^{r(g)}$ is an ideal of $\C$ it is not difficult to conclude that $ \I^{r(g)}$ is an ideal of $\I.$
Finally, it is easy to check that   $\alpha^g:\mathcal{I}^{g^{-1}}\rightarrow \mathcal{I}^g$ is an equivalence of $R$-semicategories,  for every $g\in {\rm mor}(\G).$
%

To check (v) let $f\in \alpha^{h^{-1}}({_y\mathcal{I}^h_x}\cap {_y\mathcal{I}^{g^{-1}}_x})$,  where $(g,h)\in \mathcal{G}^2.$ Then $\alpha^h(f)\in{_y\mathcal{I}^h_x}\cap {_y\mathcal{I}^{g^{-1}}_x}$ but
\begin{align*}
{_y\mathcal{I}^h_x}\cap {_y\mathcal{I}^{g^{-1}}_x}
&=({_y\I_x}\cap {_yE^h_x})\cap \beta_h({_{h^{-1}y}\I_{h^{-1}x}}\cap{_{h^{-1}y}E^{h^{-1}}_{h^{-1}x}})\cap\\
& \hspace{1,1cm}\cap ({_y\I_x}\cap {_yE^{g^{-1}}_x})\cap \beta_{g^{-1}}({_{gy}\I_{gx}}\cap {_{gy}E^g_{gx}})\\
&=({_y\I_x}\cap {_yE^h_x}\cap {_yE^{g^{-1}}_x})\cap \beta_h({_{h^{-1}y}\I_{h^{-1}x}}\cap{_{h^{-1}y}E^{h^{-1}}_{h^{-1}x}})\cap\\
&\hspace{5cm} \cap \beta_{g^{-1}}({_{gy}\I_{gx}}\cap {_{gy}E^g_{gx}})\\
&\subseteq  \beta_h({_{h^{-1}y}\I_{h^{-1}x}}\cap{_{h^{-1}y}E^{h^{-1}}_{h^{-1}x}})
 \cap \beta_{g^{-1}}({_{gy}\I_{gx}}\cap {_{gy}E^g_{gx}}).
\end{align*}

Note that,
$E^g=E^{r(g)}=E^{r(gh)}=E^{gh},$ and $E^{h^{-1}}=E^{r(h^{-1})}=E^{r((gh)^{-1})}=E^{(gh)^{-1}}.$
Hence,
\begin{align*}
f&\in  \beta_{h^{-1}}(\beta_h({_{h^{-1}y}\I_{h^{-1}x}}\cap{_{h^{-1}y}E^{h^{-1}}_{h^{-1}x}}))\cap \beta_{h^{-1}}(\beta_{g^{-1}}({_{gy}\I_{gx}}\cap {_{gy}E^g_{gx}}))\\
&  \subseteq ({_{h^{-1}y}\I_{h^{-1}x}}\cap {_{h^{-1}y}E^{(gh)^{-1}}_{h^{-1}x}})\cap \beta_{(gh)^{-1}}({_{gy}\I_{gx}}\cap {_{gy}E^{gh}_{gx}}) \\
&={_{h^{-1}y}\mathcal{I}^{(gh)^{-1}}_{h^{-1}x}},
\end{align*}
 as desired. 
\end{exe}

\section{Globalization of partial actions of groupoids}
If $\alpha$ is a partial action of a group $G$ on a  $R$-semicategory $\C,$ then there exists an globalization of $(\D,\beta),$ if and only if, for all $x\in \C_0$ and $g\in G$ the space ${_x\mathcal{I}^{g}_x}$ contains a local identity element (\cite[Theorem 4.6]{CFM}). We extend this result to the frame of partial actions of groupoids. But first,  for the sake of completeness, we recall the definition of globalization for partial actions of groupoids on algebras.
\begin{defn}\label{defgloba} A global action
$\bt=(B_g,\bt_g)_{g\in {\rm mor}( \mathcal{G})}$ of a groupoid $\G$ on a ring $B$ is
a {\it globalization} or enveloping action of a partial action
$\af=(A_g,\af_g)_{g\in{\rm mor}( \mathcal{G})}$ of $\G$ on $\
A$ if, for each $e\in
\G_0,$ there exists a ring monomorphism $\psi_e:A_e\to B_e$ such
that:
\begin{enumerate}\renewcommand{\theenumi}{\roman{enumi}}   \renewcommand{\labelenumi}{(\theenumi)}
\item $\psi_e(A_e)\unlhd B_e$;
\item $\psi_{r(g)}(A_{g})=\psi_{r(g)}(A_{r(g)})\cap\bt_g(\psi_{d(g)}(A_{d(g)}))$;
\item $\bt_g(\psi_{d(g)}(a))=\psi_{r(g)}(\af_g(a))$, for all $a\in A_{g^{-1}}$;
\item $B_g=\sum\limits_{r(h)=r(g)}\bt_h(\psi_{d(h)}(A_{d(h)}))$.
\end{enumerate}
\end{defn}
Then according to \cite[Theorem 2.1]{BP}, if  $A$ is a ring such that $A_e$ is unital, then $\alpha$ admits a globalization, if and only if, $A_g$ is unital for all $g\in {\rm mor}(G).$
 
Now we combine \cite[Definition 4.1]{CFM} and  Definition \ref{defgloba} to  give the notion of globalization of a partial groupoid on an $R$-semicategory,. 


\begin{defn}\label{defn2}
Let $\C$ be an  $R$-semicategory and $\alpha=(\mathcal{I}^g, \alpha^g)_{g\in{\rm mor}( \mathcal{G})}$ a partial action of $\mathcal{G}$ on $\C.$  We say that a pair $(\D,\beta)$, where $\D$ is an $R$-semicategory and $\beta=(\mathcal{J}^g,\beta^g)_{g\in {\rm mor}( \mathcal{G})}$ is a global action of $\mathcal{G}$ on $\D,$  is a globalization of $(\C,\alpha)$ if the following conditions are satisfied:
\begin{enumerate}
\item[(i)] $\beta_0$ is a  universal globalization of $\af_0,$ in the sense of \cite[Definition 11]{NY}.
\item[(ii)] For all $e\in \mathcal{G}_0$ there exists a faithful semifunctor $\varphi_e:\mathcal{I}^e\rightarrow \mathcal{J}^e$  such that $\varphi_e(\mathcal{I}^e)$ is an ideal of $\mathcal{J}^e.$ 
\item[(iii)] $\varphi_{r(g)}({_y\mathcal{I}^g_x})=\varphi_{r(g)}({_y\mathcal{I}^{r(g)}_x})\cap \beta_g(\varphi_{d(g)}({_{g^{-1}y}}\mathcal{I}^{d(g)}_{{g^{-1}x}}))$, for all $\{x,y\}\subseteq \C_0^{g}$ and $g\in {\rm mor}( \mathcal{G});$
\item[(iv)] $\beta_g \circ \varphi_{d(g)}(f)=\varphi_{r(g)}\circ \alpha_g (f)$, for all $f\in {_y\mathcal{I}^{g^{-1}}_x}$ and $g\in {\rm mor}( \mathcal{G});$
\item[(v)] ${_y\mathcal{J}^g_x}=\sum\limits_{r(h)=r(g)}\beta_h(\varphi_{d(h)}({_{h^{-1}y}\mathcal{I}^{d(h)}_{h^{-1}x}}))$, for all $x,y \in \D_0,$ and $g\in {\rm mor}(\G).$
\end{enumerate}
\end{defn}
\begin{rem}\label{obcon} Let
$\beta=(\mathcal{J}^g, \beta^g)_{g\in {\rm mor}(\G)}$ be a globalization for $\alpha,$ then  since $\beta_0$ is a universal globalization of $\alpha_0$ by \cite[Remark 22]{NY} we can assume that $\C_0\subseteq \D_0.$
\end{rem}
\begin{defn}
Given  $R$-semicategories $\D$ and $\D'$ with global actions $\beta=(\mathcal{J}^g,\beta^g)_{g\in {\rm mor}\mathcal{(G)}}$ and $\beta'=(\mathcal{J'}^{g},\beta'^{g})_{g\in {\rm mor}\mathcal{(G)}}.$ Suppose that $\D_0=\D'_0,$ then  we say that $\beta$ and $\beta'$ are equivalent, if  for each $e\in \mathcal{G}_0$, there is an equivalence of categories $\psi_e:\mathcal{J'}^{e}\rightarrow \mathcal{J}^e$ such that $\beta_g \circ \psi_{d(g)}(f)=\psi_{r(g)}\circ \beta'_g(f), $ for all $f\in {_y\mathcal{J'}^{d(g)}_x}, $ and $ x,y \in \D'_0$.\end{defn}

\begin{thm}\label{teogloba}
Let $\alpha=(\mathcal{I}^g, \alpha^g)_{g\in {\rm mor}\mathcal{(G)}}$ be a partial action of a groupoid $\mathcal{G}$ on an  $R$-semicategory $\C$ such that 
$
{_x\mathcal{I}^e_x}$ contains an  identity element, for any $x\in \C_0$ and each $e\in \mathcal{G}_0$. Then, $\alpha$ admits a globalization $\beta$ if and only if each $R$-space ${_x\mathcal{I}^g_x}$  contains a local identity element, for each $g\in {\rm mor}\mathcal{(G)}$, and each $x\in \C_0$. Furthermore, if $\beta$ exists then it is unique up to equivalence.
\end{thm}

\begin{proof} To show ($\Rightarrow$).
Let $\beta=(\mathcal{J}^g, \beta^g)_{g\in {\rm mor}(\G)}$  be a globalization for $\alpha$ and $\varphi_e:\mathcal{I}^e\rightarrow \mathcal{J}^e,\, e\in \mathcal{G}_0$  be the functor given by  (ii) in Definition \ref{defn2}. By  Remark \ref{obcon} we can assume that $\C_0\subseteq \D_0.$ 
 Take  $x\in \C_0,$ if $x\notin \C_0^{g},$ then  ${_x\mathcal{I}^g_x}=\{0\}$ and clearly has a local identity element. Now if $x\in \C_0^{g},$ then

$$\varphi_{r(g)}({_x\mathcal{I}^g_x})=\varphi_{r(g)}({_x\mathcal{I}^{r(g)}_x})\cap \beta_g(\varphi_{d(g)}({_{g^{-1}x}}\mathcal{I}^{d(g)}_{{g^{-1}x}}))$$
which implies that  ${_x\mathcal{I}^g_x}$ has a local identity element.

To show ($\Leftarrow$), assume that ${_x\mathcal{I}^g_x},$ for  $ g\in {\rm mor}(\mathcal{G}), $ and $x\in \C_0$ contains a local  identity element ${_x1^g_x}$. Consider first  $\beta_0=(Y, \beta_0^g)_{g\in {\rm mor}(\G)}$  a universal globalization of $\alpha_0.$ For $y\in Y$ and  $g\in \G^y,$ write  $\beta_0^g(y)=gy.$  Define the category $\mathfrak{F}$ as follows: $\mathfrak{F}_0=\C_0$ and for any $x,y \in \mathfrak{F}_0,$ we set   

$${_y\mathfrak{F}_x}=\left\{f:\G\to \prod_{g\in \G^y\cap \G^x}{_{gy}\C_{gx}}\mid  f(l)\in {_{l^{-1}y}\C_{l^{-1}x}}, \,\text{for each $l^{-1}\in\G^y\cap \G^x$}\right\},$$
where  $_{y}\C_{v}=\{0\}$ if $\{v,y\}$ is not a subset of $\C_0.$

Take $g\in {\rm mor}(\mathcal{G})$ and set 
$F^g=\{f\in \,_y\mathfrak{F}_x\mid f(h)=0,\forall h\notin \G(-,r(g))\}$. As in \cite{CFM}, the composition of morphisms is defined by $(k \circ l)(h)=k(h) \circ l(h), $ for all $ h\in {\rm mor}(\G).$
Note hat  $F^g$ is an ideal of $\mathfrak{F}$  
such that $F^g=F^{r(g)}$ and $F^g_0=\mathfrak{F}_0,$ for any $g\in {\rm mor}(\G). $ As usual, we denote the value $f(h)\in $ by $f|_h$, for all $f\in\,_y\mathfrak{F}_x$ and $h\in {\rm mor}\mathcal{(G)}$.  Now 
for $g\in {\rm mor}\mathcal{(G)}$ and $f\in F^{g^{-1}}$ let $\beta_g:F^{g^{-1}}\rightarrow F^g$ be the map given by
\[\beta_g(f)|_h=
  \begin{cases}
  f(g^{-1}h), & \text{if $h\in  \G(-,r(g))$}\\
  0, & \text{otherwise.}
  \end{cases}\]
As in the proof of \cite[Theorem 2.1]{BP}
one can show that $\beta_g$ is well defined and $\beta=(F^g, \beta_g)_{{\rm mor}(\mathcal{G})}$ is an action of $\mathcal{G}$ on $\mathfrak{F}$. 

Now, for each $e\in \mathcal{G}_0$, we define $\varphi_e:{\mathcal{I}^e}\rightarrow {F^e},$ as a map  $\varphi_e:{\mathcal{I}^e_0}\rightarrow {F^e_0},$ $\varphi_e$ is the inclusion. Moreover, $\varphi_e:{_y\mathcal{I}^e_x}\rightarrow {_yF^e_x}$ given by
\[\varphi_e(a)|_h=
  \begin{cases}
  \alpha^{h^{-1}}(\eta\,{_x1^h_x}), & \text{if $\{x,y\}\subseteq \I^e$ and $r(h)=e$}\\
  0, & \text{otherwise.}
  \end{cases}\]
For all $\eta \in {_y\mathcal{I}^e_x},$ and  $h\in{ \rm mor}\mathcal{(G)}$. The proof that  $\varphi_e:{_y\mathcal{I}^e_x}\rightarrow {_yF^e_x}$ is a faithful semifunctor is similar to  the one presented in  \cite[Theorem 4.6]{CFM}. 

Now for each morphism $g$ in $\G$ we consider  $E^g$ as  the subcategory of $F^g$ defined as follows: the set of objects $E^g_0$ of $E^g$ is equal to $\C_0$ and for $x,y\in E_0$  the set of morphisms from $x$ to $y$ is given by
$${_yE^g_x}=\sum_{r(h)=r(g)}\beta_h(\varphi_{d(h)}({_{h^{-1}y}\mathcal{I}^{d(h)}_{h^{-1}x}})),$$ for all $g\in \mathcal{G},$ where  ${_{h^{-1}y}}\mathcal{I}^{d(h)}_{h^{-1}x}=\{0\}$ if $\{x,y\}$ is not a subset of $\C_0.$ Following the assumptions given in the  proof of  \cite[Theorem 2.1]{BP} we consider the product $R$-semicategory  $\mathcal{T}=\prod_{e\in \mathcal{G}_0}E^e,$  and  $\iota_e: E^e\to T$ be the injective functor given by $\iota_e(x)=(x_l)_{l\in \mathcal{G}_0}$, with $x_e=x$ and $x_l=0$ for all $l\ne e$. Also, we  identify $E^e$ with $\iota_e(E^e)$ and $\varphi_e$   with $\iota_e \circ \varphi_e,$ we will denote also by the same $\beta_g$, given by $\iota_{r(g)}\circ \beta_g\mid_{E^{g^{-1}}}\circ \iota^{-1}_{d(g)}$ from $\iota_{d(g)}(E^{g^{-1}})\equiv E^{g^{-1}}$ onto $E^g\equiv \iota_{r(g)}(E^g)$.\\
Then 
$\beta=(E^g, \beta_g)_{g\in mor(\mathcal{G})}$ is a global action of $\mathcal{G}$ on $\mathcal{T}$. 
We need to show that $\beta$ is a globalization of $\alpha$. By our  construction the conditions (i) and (v) of  Definition \ref{defn2} are satisfied. Also, the proof of  properties  (ii), (iii) and (iv) are analogous to the group case (see the proof of   \cite[Theorem 4.6]{CFM}). 

To end the proof it is required to show the uniqueness (up to equivalence) of the globalization $\beta$ of $\alpha$. Now, suppose that $\beta'=(\mathcal{J}^g,\theta ^g)_{g\in {\rm mor}(\mathcal{G})}$ is a global action of $\G$ on  $\J$ with $\mathcal{T}_0=\J_0$  which is also a globalization of $\alpha.$ Then, for each $e\in \mathcal{G}_0$ there are   faithful semifunctors $\varphi'_e\colon \I^e\to \J^e$
 with $ {_y\mathcal{J'}^g_x}=
\sum_{r(h)=r(g)}\theta_h(\varphi'_{d(h)}({_{h^{-1}y}\mathcal{I}^{d(h)}_{h^{-1}x}})).$
We define the semifunctors $\eta_e:\J'^{e}\rightarrow E^e$ as the identity in the objects and $\eta_e:\,_y\J'^{e}_x\rightarrow\, _yE^e_x$  by 
$$\sum_{i=1}^n\beta'_{h_i}(\varphi'_{d(h_i)}(a_i))\mapsto \sum_{i=1}^n\beta_{h_i}(\varphi_{d(h_i)}(a_i)),$$
with $h_i\in \G(-,e)$ and $a_i\in {_{y}\mathcal{I}^{d(h_i)}_{x}}$, for all $1\leq i\leq n$.\\
As in the last part of the proof of Theorem 2.1 in \cite{BP}, it follows that $\eta_e$ is well defined and so it is an isomorphism semicategories. This completes the proof.
\end{proof}

\begin{exe}
We consider the $R$-semicategory $\mathcal{C}$ from Example \ref{exe1}. Note that $_u{\mathcal I}^g_v=_u{\mathcal I}^{r(g)}_v, \,_u{\mathcal I}^{g^{-1}}_v$ and $_u{\mathcal I}^{d(g)}_v$ have local identity elements. Then by Theorem \ref{teogloba}, there is an $R$-category $\mathcal{D}$ and a global action  $\beta=(E^g, \beta^g)_{g\in mor({\mathcal{G}})}$ which is a globalization of $\alpha$, where ${\mathcal D}_0={\mathcal C}_0$, that is, $\mathcal D^{r(g)}_0=\mathcal D^g_0=\mathcal D^{d(g)}_0=\mathcal D^{g^{-1}}_0=\C_0$, $\beta^h_0=id_{D^h_0}$, for  $h\in \{d(g), r(g)\}$  and  $\beta^g_0=\beta^{g^{-1}}_0: x\mapsto y, y\mapsto x.$
 Now for  $u,v\in {\mathcal D}_0$, $E^{d(g)}=E^{g^{-1}}={_u{\mathcal I}^{d(g)}_v}, \, E^{r(g)}=E^g={_u{\mathcal I}^{r(g)}_v}\oplus R$ and ${_u{\mathcal D}_v}={_u{\mathcal C}_v}\oplus Re_4=Re_1\oplus Re_2 \oplus Re_3 \oplus Re_4$ with  $e_1, e_2, e_3, e_4$ pairwise orthogonal central idempotents with sum $1_{_u{\mathcal D}_v}$, moreover
\begin{align*}
\beta^{d(g)}={\rm id}_{E^{d(g)}},& \quad \beta^{r(g)}={\rm id}_{E^{r(g)}},\\
\beta^g(ae_1+be_2)&=ae_3+be_4,\\
\beta^{g^{-1}}(ae_3+be_4)&=ae_1+be_2,
\end{align*}
for all $a,b \in R$.
\end{exe}

\begin{rem}\label{cfdp} Let $\alpha=(\mathcal{I}^g, \alpha^g)_{g\in {\rm mor}\mathcal{(G)}}$ be a partial action of a groupoid $\mathcal{G}$ on a  $R$-semicategory $\C$ having a globalization  $\beta=(\mathcal{J}^g, \beta^g)_{g\in {\rm mor}\mathcal{(G)}},$ 
\begin{itemize}
\item Assuming that $\mathcal{I}^e\subseteq \mathcal{J}^e$ for all $e\in \G_0.$ One has that $\beta_{\G_e}=(E^g, \beta^g)_{g\in \G_e}$ acts globally on the $R$-semicategory $E^e,$ where $E^e_0=\C_0,$ and ${_yE^g_x}=\sum\limits_{r(h)=e}\beta_h({_{h^{-1}y}\mathcal{I}^{d(h)}_{h^{-1}x}}),$ for all $g\in \G_e$ and $x,y\in \C_0.$ Then the action of $\G_e$ on the $R$-semicategory $\mathcal{E}^e$ of $E^e,$ where $\mathcal{E}^e_0=\C_0,$ and  for all $g\in \G_e$ and $x,y\in \C_0,$ we have ${_y\mathcal{E}^g_x}=\sum_{h\in \G_e}\beta_h({_{h^{-1}y}\mathcal{I}^e_{h^{-1}x}}),$  is a globalization of the partial action $\alpha_{\G_e}$ of $\G_e$ on  $\mathcal{I}^e,$ (see \cite[Definition 4.1]{CFM}), then  Theorem \ref{teogloba} is a generalization of \cite[Theorem 4.6]{CFM}.
\item In the case that  $\C_0=\{x\}$ then  follows that  $\C$ is an $R$-algebra and $\alpha_0$ has the identity in $\C_0$ as an enveloping action,  thus  \cite[Theorem 2.1]{BP}  is a consequence  Theorem \ref{teogloba}.
\end{itemize}
\end{rem}

\section{The partial skew groupoid semicategory}

In this section we introduce the definition of partial skew groupoid semicategory, we give a sufficient associativity condition   and  show an isomorphism between algebras associated to them. 

\begin{defn}
Let $\alpha=(\I^g, \alpha^g)_{g\in {\rm mor}{(\G)}}$ be partial action of a groupoid $\mathcal{G}$ on a  $R$-semicategory $\C.$ We define the skew groupoid (non-necessarily associative)  semicategory $C\ast_{\alpha}\mathcal{G}$ by:
\begin{enumerate}
\item[(i)] $(\C\ast_{\alpha}\mathcal{G})_0=\C_0$.
\item[(ii)] For each $x,y \in \C_0, $ we set ${_y(C\ast_{\alpha}\mathcal{G})_x}=\bigoplus\limits_{g\in \G^x}{_y\mathcal{I}^g_{gx}}$.
\end{enumerate}

For  $t,g\in {\rm mor}{(\G)}$ and $x,y\in \C_0^{g^{-1}}\cap \C_0^{t^{-1}}$ 
we define the product of  $f\in {_z\mathcal{I}^t_{ty}}$ and $ l\in {_y\mathcal{I}^g_{gx}}$  by the rule
\[fl=
  \begin{cases}
  \alpha^t(\alpha^{t^{-1}}(f) l)\in {_z\mathcal{I}^{tg}_{(tg)x}} & \text{if $(t,g)\in \mathcal{G}^2$} \,\, \,and \,\,\,x\in \C_0^{(tg)^{-1}},\\
  0, & \text{otherwise.}
  \end{cases}\]
\end{defn}

\begin{exe}
We consider 
the fo\-llo\-wing R-semicategory $\C$ with
\begin{enumerate}
\item $\C_0=\{x, y\}.$
\item Given $ u,v \in \C_0$ let ${_u\C_v}=Re_1\oplus Re_2\oplus Re_3\oplus Re_4,$  where $e_1,e_2,e_3, e_4$ are pairwise orthogonal central idempotents with sum 1.
\item For all $u,v,w \in \C_0$ an $R$-bilinear map $\cdot: {_u\C_v}\times {_v\C_w}\to {_u\C_w};$ given by multiplication.
\end{enumerate}
Take the groupoid $\G=\{g_1, g_2, g_3\}\, \,\text{with}\, \,\G_0=\{g^{-1}, g_2\}, \, g^{-1}_3=g_3, \, g_3g_3=g_2$.   Then 
 $\G$ acts partially on $\C_0$ via $\af_0,$ where: 
$$\C_0^{g_1}=\{x\},\,\,\, \C_0^{g_2}=\{y\}, \,\,\,\text{and}\,\,\, \C_0^{g_3}=\C_0$$ 
and $\alpha_0^{g_3}:\C_0^{g_3}\to \C_0^{g_3}: \, y \mapsto x \, \text{and}\, x\mapsto y ;\, \alpha_0^{g_1}={\rm id}_{\C_0^{g_1}}$ and $\alpha_0^{g_2}={\rm id}_{\C_0^{g_2}}$.\\
Now consider the ideals of $\C$  given by 
\begin{itemize}
\item $_u{\mathcal I}^{g_1}_v={_u\C_v}$ if $(u,v)=(x,x)$ and  $_u{\mathcal I}^{g_1}_v=0,$ if $(u,v)\neq (x,x). $
\item $_u\mathcal{I}^{g_2}_v=Re_2\oplus Re_3\oplus Re_4$ if $(u,v)=(y,y)$ and $_u\mathcal{I}^{g_2}_v=0,$ if $(u,v)\neq (y,y). $
\item $_u\mathcal{I}^{g_3}_v=Re_2\oplus Re_3,$ for all $u,v\in \C_0.$
\end{itemize}

Then the family $\alpha=(\mathcal{I}^g, \alpha^g)_{g\in {\rm mor}(\G)}$ is a partial action of $\G$ on $\C,$ where   $\alpha^{g_3}(ae_2+be_3)=be_2+ae_3, \, \alpha^{g_1}=id_{\mathcal{I}^{g_1}}, \, \alpha^{g_2}=id_{\mathcal{I}^{g_2}},$ for each $a, b\in R$.

We describe the skew category $\C\ast_{\alpha}\mathcal{G}$ thus
\begin{enumerate}
\item[(i)] $(\C\ast_{\alpha}\mathcal{G})_0=\{x,y\}$.
\item[(ii)] Finally for each $x,y \in \C_0:$\begin{align*}
{_x(\C\ast_{\alpha}\mathcal{G})_x}&={_x{\mathcal I}^{g_1}_x}\oplus{_x{\mathcal I}^{g_3}_y},\\
{_x(\C\ast_{\alpha}\mathcal{G})_y}&={_x{\mathcal I}^{g_2}_y}\oplus{_x{\mathcal I}^{g_3}_x},\\
{_y(\C\ast_{\alpha}\mathcal{G})_y}&={_y{\mathcal I}^{g_2}_y}\oplus{_y{\mathcal I}^{g_3}_x},\\
{_y(\C\ast_{\alpha}\mathcal{G})_x}&={_y{\mathcal I}^{g_1}_x}\oplus{_y{\mathcal I}^{g_3}_y}.
\end{align*}
\end{enumerate}

\end{exe}

\begin{defn}
Let $\mathcal{G}$ be a groupoid and $\alpha=(\I^g, \alpha^g)_{g\in {\rm mor}(\mathcal{G})}$ a partial action of $\mathcal{G}$ on an $R$-semicategory $\C.$ We say that $\alpha$ is  associative if the composition of maps in $\C\ast_{\alpha}\mathcal{G}$  associative.
\end{defn}
By a routine calculation one can show that every global action is associative.
\begin{rem}
As a consequence of the definition above, if $\C$ is  R-semicategory and the partial action $\alpha$ is    associative, then $\C\ast_{\alpha}\mathcal{G}$ is a  $R$-semicategory and we call it the partial skew groupoid semicategory.
\end{rem}

For further reference, we give the following.
\begin{lem} \label{skewcat} Let $\alpha=(\I^g, \alpha^g)_{g\in {\rm mor}{(\G)}}$ be an associative  partial action of a groupoid $\mathcal{G}$ on an  $R$-semicategory $\C.$ Suppose that $\G_0$ is finite and that  for any  $x\in \C_0$ and $e\in \G_0\cap \G^x$ the ideal $_x\I^e_x$ has a local identity. Then   $\C\ast_{\alpha}\mathcal{G}$ is a  $R$-category . 
\end{lem}
\begin{proof} Let $x\in \C_0,$ $e\in \G_0\cap \G^x$  and  $_x1^e_x\in\, _x\I^e_x$ be a local identity. Write $u=\sum_{e\in \G^x}{_x1^e_x},$ then for  $z\in \G_0$ and $f\in \,_z\I^t_{tx}\in\, _z\C\ast_{\alpha}\mathcal{G}_x,$ we shall check that $fu=f,$ which is clear if $f=0.$ Now if $f\neq 0,$ then $x\in \C_0^{t^{-1}}\subseteq \C_0^{d(t)}$  thus $d(t)\in \G^x$ and $fu=\alpha^t(\alpha^{t^{-1}}(f) _x1^{d(t)}_x)=\alpha^t(\alpha^{t^{-1}}(f))=f.$ Analogously one shows that $ug=g$ for each $g\in \,_x(\C\ast_{\alpha}\mathcal{G})_z.$\end{proof}

\subsection{The multiplier ring and the  associativity  of $\C\ast_{\alpha}\mathcal{G}.$  }
Now we are interested in the associativity of $\C\ast_{\alpha}\mathcal{G}.$ For this we need to recall the mulriplier algebra associated to a ring.
 
Let $A$ be  a non-necessarily unital ring.
As in \cite{DRS}, for homomorphisms of left $A$-modules we use the right-hand side notation. That is,  given a left $A$-module homomorphism 
$\gamma\colon _A M\to _A\hspace*{-0.17cm} N$ and $x\in M$ we write $x\gamma$ instead of $\gamma x$; while for homomorphisms of right $A$-modules we use the usual notation. Thus, we read composition of left module homomorphism from left to right, and we read composition of right module homomorphism in the usual right to left way.
\smallbreak

Let $A$ be a ring. The multiplier ring $M(A)$ of $A$ is the set
\[M(A)=\{(R,L)\in {\rm End}(_A A)\times {\rm End}(A_A) \mid (aR)b=a(Lb)\, \text{ for all } a,b\in A\},\]
with component-wise multiplication and addition; (see e.~g. \cite[Section 2]{DE} for details). For a multiplier $\gamma=(R,L)\in M(A)$ and $a\in A$ we set: $a\gamma=aR$ and $\gamma a=La$. Consequently,
\begin{equation*}\label{asmul}(a\gamma)b=a(\gamma b),\end{equation*} for all $a,b\in A$. Also, an element $a\in A$ determines the multiplier $(R_a,L_a)\in M(A)$, where $xR_a=xa$ and $L_ax=ax$, for all $x\in A$.

\begin{defn}\cite[Definition 2.4]{DE} We say that $A$ is $(\Ll,\R)$ associative, if given any two multipliers $(L,R)$ and $(L',R')$ one has that $R'\circ L = L'\circ R.$
\end{defn}

Let $\C$ be an  $R$-semicategory, in \cite{CM} the authors introduced the $R$-algebra $a(\C)=\bigoplus_{x,y\in \C_0}{_y\C_x}$ provided with the matrix product induced by the composition in $\C$.   Notice that if $\C$ is a category with a finite number of objects, then $a(\C)$ is a unital $R$-algebra.

\begin{defn} We say that $\C$ is $(\Ll,\R)$ associative if $a(\C)$ is $(\Ll,\R)$ associative.
\end{defn}

The following result gives a necessary condition for the associativity of  $\C\star_{\alpha}\mathcal{G}.$

\begin{prop}\label{assoc}
If $\alpha=(\I^g, \alpha^g)_{g\in {\rm mor}{(\G)}}$ be partial action of a groupoid $\mathcal{G}$ on a  $R$-semicategory $\C$ such that  $\mathcal{I}^g$ is $(\mathcal{L},\mathcal{R})-$associative for every $g\in \mathcal{G}$, then the partial skew groupoid category $\C\star_{\alpha}\mathcal{G}$ is associative.
\end{prop}

\begin{proof}
Now, for to prove that $\C\star_{\alpha}\mathcal{G}$ is associative is enough to verify that
\begin{equation}\label{e1}
(f l) k=f (l k),
\end{equation}
for any non-zero  morphisms $f,l$ and $k$ that are composable. Write $f\in {_z\mathcal{I}^t_{ty}},  l\in {_y\mathcal{I}^g_{gx}}$ and  $ k\in {_x\mathcal{I}^h_{hu}},$ where $z,y, x,u\in \C_0$ and $t,g,h$ are morphisms in $\C.$  Note that if $(t,g)\notin \mathcal{G}^2$ (resp. $(tg,h)\notin \mathcal{G}^2$) then $(t,gh)\notin \mathcal{G}^2$ (resp. $(t,h)\notin \mathcal{G}^2$) and in this case (\ref{e1}) follows trivially. Thus, assume that $(t,g)$ and $(tg,h)\in \mathcal{G}^2$. The left-hand side of (\ref{e1}) equals to 
\begin{equation*}
(fl) k=\alpha^{tg}(\alpha^{(tg)^{-1}}(\alpha^t
(\alpha^{t^{-1}}(f)l))k).
\end{equation*}
But $\alpha^t
(\alpha^{t^{-1}}(f) l)\in \alpha^t(\mathcal{I}^{t^{-1}}\cap \mathcal{I}^g)\subseteq \mathcal{I}^{tg}$
and by $(iv)$ of the Definition  \ref{apgc} we have
\begin{align*}
\alpha^{(tg)^{-1}}(\alpha^t
(\alpha^{t^{-1}}(f)l))&=\alpha^{g^{-1}}(\alpha^{t^{-1}}(\alpha^t(\alpha^{t^{-1}}
(f) l)))\\
&=\alpha^{g^{-1}}(\alpha^{t^{-1}}(f)l).
\end{align*}
Hence $(fl)k$ equals to $
\alpha^{tg}(\alpha^{g^{-1}}(\alpha^{t^{-1}}(f) l))=\alpha^t(\alpha^g(\alpha^{g^{-1}}(\alpha^{t^{-1}}(f)l))).$

Calculating the right hand side of \eqref{e1} we get
\begin{align*}
f (l k)&=f\circ (\alpha^{g}(\alpha^{g^{-1}}(l)k))\\
&=\alpha^t(\alpha^{t^{-1}}(f) (\alpha^g(\alpha^{g^{-1}}(l)k))).
\end{align*}
And applying $\alpha^{t^{-1}}$ we have that \eqref{e1} holds if and only if
\begin{equation*}
\alpha^g(\alpha^{g^{-1}}(\alpha^{t^{-1}}(f)l) k)=\alpha^{t^{-1}}(f)(\alpha^g(\alpha^{g^{-1}}(l)k)),
\end{equation*}

Now note that $\alpha^{t^{-1}}(f)$ runs over all the elements of $_{t^{-1}z}{\I}^{t^{-1}}_y$. Consequently, \eqref{e1} is equivalent to the following:
\begin{equation}\label{e3}
\alpha^g(\alpha^{g^{-1}}(f'l)k)=f \alpha^g(\alpha^{g^{-1}}(l) k).
\end{equation}
for all $g, t, h$ morphisms  in $ \mathcal{G}$ such that $(t,g), (tg,h)\in \mathcal{G}^2$ and all $f'\in _{t^{-1}z}{\I}^{t^{-1}}_y,  l\in {_y\mathcal{I}^g_{gx}}$ and  $ k\in {_x\mathcal{I}^h_{hu}},$ where $z,y, x,u\in \C_0$.\\
Having in mind that  that $\mathcal{I}^{t^{-1}}\subseteq \mathcal{I}^{d(t)}=\mathcal{I}^{r(g)},$ $ \mathcal{I}^g$ is an ideal of $\mathcal{I}^{r(g)}$ and $\mathcal{I}^{g^{-1}}$ is an ideal of $\mathcal{I}^{d(g)}=\mathcal{I}^{d(tg)}=\mathcal{I}^{r(h)}$. Then the restriction of $\mathcal{R}_{f}$  (resp. $\mathcal{R}_{k}$) to $\mathcal{I}^g$  (resp. $\mathcal{I}^{g^{-1}}$) is a right multiplier of $M(\mathcal{I}^g)$ (resp. of $M(\mathcal{I}^{g^{-1}})$) and, consequently, (\ref{e3}) is equivalent to  the equality
\begin{equation}\label{e4}(\alpha^g \circ \mathcal{R}_{k} \circ \alpha^{g^{-1}})\circ \mathcal{L}_{f'}=\mathcal{L}_{f'}\circ (\alpha^g\circ \mathcal{R}_{k}\circ \alpha^{g^{-1}}),\end{equation}
is valid on $\mathcal{I}^g,$ for all $g$ morphism in $ \mathcal{G}$ and  $f'\in _{t^{-1}z}{\I}^{t^{-1}}_y$ and $ k\in {_x\mathcal{I}^h_{hu}}$.  However the last relation holds since $\alpha^g \circ \mathcal{R}_{k} \circ \alpha^{g^{-1}}$ is a right multiplier of $M(\mathcal{I}^g)$ (thanks to  \cite[Proposition 2.7]{DE}) and by the assumption that  $\mathcal{I}^h$ is $(\mathcal{L}, \mathcal{R})$-associative for any $h\in \mathcal{G}.$
\end{proof}
\medskip

Given a partial action $\af=(D_g, \af_g)_{g\in {\rm mor }(G)}$ of a groupoid $\G$ on a ring $A$, we recall from \cite{BFP} that the {\it partial skew groupoid ring} is the direct sum $$A*_\alpha \G:=\bigoplus_{g\in {\rm mor }(G)}D_g\delta_g,$$
where the $\delta_g '$s are symbols, with the usual addition and the multiplication given by
\[
(a_g\delta_g)(b_h\delta_h)=\left\{\begin{array}{lc}
                       \alpha_g(\alpha_{g^{-1}}(a_g)b_h)\delta_{gh}, & \text{if } d(g)=r(h) \\
                       0, & \text{otherwise}
                     \end{array}\right.\]
for all $g,h\in \G$, $a_g\in D_g$ and $b_h\in D_h$. 
The following results states that the algebra associated to the partial skew groupid category is a partial skew groupoid ring.

\begin{prop}\label{cattoalg}
Let $\alpha=(\mathcal{I}^g, \alpha^g)_{g\in {\rm mor}( \mathcal{G})}$ be an associative  partial action of a groupoid $\mathcal{G}$ on an  $R$-semicategory $\C$. 
Then $\mathcal{G}$ acts partially on $a(\C)$ 
 and the  $R$-algebras $a(\C\ast_{\alpha}G)$ and  $a(\C)\ast_{\alpha}\mathcal{G}$ are  isomorphic.
\end{prop}

\begin{proof}
For each $g\in {\rm mor}( \mathcal{G})$ let $a(C)_g=\bigoplus_{x,y\in \C_0}{_y\mathcal{I}^g_{x}}$. Note that as ${\mathcal{I}^{r(g)}}$ is an  ideal of $\C$ and ${\mathcal{I}^g}$ is a ideal of ${\mathcal{I}^{r(g)}}$, then
$a(\C)_{r(g)}$ is an ideal of $a(\C)$ and $a(\C)_g$ is an ideal of $a(C)_{r(g)}$,  $\alpha_g:a(\C)_{g^{-1}}\rightarrow a(\C)_g$, defined by $\alpha_g|_{{_y\mathcal{I}^{g^{-1}}_{x}}}=\alpha^g|_{{_y\mathcal{I}^{g^{-1}}_{x}}}$, for all $x,y\in \C^{g^{-1}}_0$ and extended to $a(\C)_{g^{-1}}$ by linearity, is an isomorphism of ideals.\\
Now we show that $\alpha_{a(\C)}=(a(c)_g, \alpha_g)_{g\in {\rm mor}( \mathcal{G})}$ is a partial action of $\mathcal{G}$ on $a(\C)$. The first condition in  Definition \ref{BP}  is obvious. For the second condition, suppose that
$f\in \alpha_{h^{-1}}(a(\C)_h\cap a(\C)_{g^{-1}})$. As ${_yI_x}={_y\C_x} \cap I$, we can assume that $f\in\, _y\C_x$, so $\alpha_h(f)\in {_{hy}\C_{hx}}$ and consequently $f\in a(\C)_{(gh)^{-1}}$. Finally the condition (iii) of Definition \ref{BP}  is also clear. 

For the second assertion, we define $\varphi:a(\C\ast_{\alpha}G)\rightarrow a(\C)\ast_{\alpha}\mathcal{G}$ by $\varphi(f_g)=f_g\delta_g$, where $f_g$ is a morphism in ${_y\mathcal{I}^{g}_{gx}}\subseteq {_y(\C\ast_{\alpha}\mathcal{G})_x}$
We  have that $\varphi$ is a well defined homomorphism of $R$-algebras. Finally, $\psi:a(\C)\ast_{\alpha}\mathcal{G}\rightarrow a(\C\ast_{\alpha}\mathcal{G})$ defined by $\psi(f_g\delta_g)=f_g$ for any $f_g\in a(\C)_g$ is an inverse of $\varphi$.
\end{proof}

Recall that a ring $A$ is called (left) $s$-unital, if for all $x\in A$ one has that $x\in Ax.$ We have the following.
\begin{prop}\label{asu} Let $\C$ be an $R$-semicategory such that $\C_0$ is finite and for each $u\in \C_0$ there is an ideal $\I(u)$ of $\C$ such that  $_u\I(u)_u$ contains a left local identity. Then $a(\C)$ is a left s-unital ring.
\end{prop}
\begin{proof} Let $\omega\in a(\C)$ and  write $\omega=\sum\limits_{x,y\in \C_0}f_{y,x},$ for some $f_{y,x}\in\, _y\C_x.$ Take   $y\in \C_0,$ then by assumption there exists an ideal $_yI_y$ of $\C$ and  $e_y\in \,_yI_y$ such that $e_yf_{y,x}=f_{y,x},$ for all $x\in \C_0$ Thus in the ring  $a(\C)$ we have that 
\begin{equation}\label{sun}e_y\sum_{x\in \C_0}f_{y,x}=\sum_{x\in \C_0}f_{y,x},\,\,\,\text{and}\,\,\,e_y\sum_{x,y'\in \C_0\atop y\neq y'}f_{y',x}=0.
\end{equation}
Write $e=\sum_{y\in \C_0}e_y,$ it follows by \eqref{sun} that $e\omega=\omega,$ which implies that $a(\C)$ is a left $s$-unital ring.
\end{proof}

\begin{rem} 

Let  $\alpha_{a(\C)}=(a(\C)_g, \alpha_g)_{g\in {\rm mor}( \mathcal{G})}$ be the partial action of $\G$ on $a(\C)$ induced by $\alpha$ as in Proposition \ref{cattoalg},  it is clear that if $\alpha$ is global then so is $\alpha_{a(\C)}.$
\end{rem}

\begin{thm}\label{morequiv}
Let  $\C$ and  $\mathcal{T}$ be  $R$-categories,  if $\alpha=(\mathcal{I}^g, \alpha^g)_{g\in {\rm mor}(\G)}$  and $\beta=(\mathcal{J}^g, \beta^g)_{g\in {\rm mor}(\G)}$ are partial groupoid actions of $\mathcal{G}$ on $\C$ and $\mathcal{T},$ respectively, such that $(\mathcal{T}, \beta)$  is a globalization of $\alpha$. Then   $(\beta_{a(\mathcal{T})}, \mathcal{G}) $ is a globalization of $\alpha_{a(\C)}.$ In particular if $\alpha$ and $\beta$ are associative, then  $a(\C\star_{\alpha}\mathcal{G})$ and $a(\mathcal{T}\star_{\beta}\mathcal{G})$ are Morita equivalent. 
\end{thm}
\begin{proof} For each $g\in {\rm mor}( \mathcal{G})$ let $a(\C)_g=\bigoplus\limits_{x,y\in \C_0}{_y\mathcal{I}^g_{x}},$  In particular, for $e\in \G_0$ we get $a(\C)_e=\bigoplus\limits_{x,y\in \C_0}{_y\mathcal{I}^e_{x}},$ now by (ii) of Definition \ref{defn2} there exists a faithful semifunctor $\varphi_e:\mathcal{I}^e\rightarrow \mathcal{J}^e,$ which by the proof of Theorem \ref{teogloba} can be considered as the inclusion  $\varphi_e:\mathcal{I}_0^e\rightarrow \mathcal{J}_0^e,$  then $\varphi_e:\bigoplus\limits_{x,y\in \C_0}{_y\mathcal{I}^g_{x}}\to \bigoplus\limits_{x,y\in \C_0}{_{y}\mathcal{J}^g_{
x}}$ is a ring monomorphism, and thus  $\varphi_e:\bigoplus\limits_{x,y\in \C_0}{_y\mathcal{I}^g_{x}}\to \bigoplus\limits_{x,y\in \mathcal{T}_0}{_y\mathcal{J}^g_{x}}$ is a ring monomorphism. Now it is clear that items (i)-(iv) in Definition \ref{defgloba} follow from itens (iii)-(vi) in Definition \ref{defn2}. The last assertion follows from Proposition \ref{asu}, Proposition \ref{cattoalg} and \cite[Theorem 4.5]{BPi}.
\end{proof}

\section{The smash product semicategory} 
We start this section by giving the construction of a quotient $R$-semicategory.

Let $\alpha=(\I^g,\beta^g)_{g\in {\rm mor}( \mathcal{G})}$ be a partial action on an $R$-semicategory $\C,$ and let $\alpha^0$ be the corresponding partial action on $\C_0,$  write $ge=\af_g^0(e),$ for each $e\in \C_0^{g^{-1}}.$ We say that   $\C$ is a free $\G$-semicategory if  
\begin{equation}\label{gfree} \text{For any}\,\, e\in \C_0^{g^{-1}}\cap \C_0^{h^{-1}}\,\,\, ge=he\,\,\, \text{implies}\,\,\,\,  g=h.\end{equation} 
\begin{rem}
The relation  $e\sim f,$ if and only if,  exists $g\in {\rm mor}( \mathcal{G})$ such that $e \in \C_0^{g^{-1}}$ and $gx=y.$ is not an equivalence relation in $\C_0.$ Indeed it is reflexive and symmetric  but  not necessarily transitive. Let $\cong$ be the transitive clausure of $\sim,$ then $e\cong f$ if and only if there are $n\in \Z^+,$ $g_1,\cdots, g_n\in {\rm mor}( \mathcal{G})$ such that $$e\in \C_0^{g_1^{-1}}, g_1e\in \C_0^{g_2^{-1}}, \cdots, g_{n-1}(\cdots ( g_2 (g_1 e))\cdots)\in \C_0^{g_n^{-1}}$$ and 
$$g_n(g_{n-1}( g_2 (g_1 e))\cdots))=f.$$  
In the case that $\G$ is a group or that $\alpha^0$ is given by multiplication (see Example \ref{exe3}), then $\sim$ is an equivalence relation, which is known as  the orbit equivalence  relation  determined by $\af.$  
\end{rem}
 Now consider the $R$-category $\mathcal{C}/\G$ whose set of  objects are the equivalence classes determined by $\cong.$  To define the morphisms of $\mathcal{C}/\G$   
Let  $\tau , \rho\in \mathcal{C}_0/\cong$ and consider   the external direct sum 
\begin{equation}\label{cmor}C(\rho,\tau)=\bigoplus\limits_{x\in \rho, y\in \tau}{_y}\C_x,\end{equation}
Then there is an $R$-bilinear  map $C(\rho,\tau)\times C(\tau, \kappa)\to C(\rho, \kappa)$ provided with the matrix product induced by the composition in $\C.$  
We  have the following. 

\begin{prop}\label{orbita} Let    $\C$  be  a free $\G$-semicategory and $\tau , \rho\in \mathcal{C}_0/\cong.$ Then the  family  $\alpha=(C(\rho,\tau)_g,\af_g)_{g\in {\rm mor}( \mathcal{G})},$ where $C(\rho,\tau)_g=\bigoplus\limits_{x\in \rho, y\in \tau}{_y}\I^g_x $ 
$$\af_g\colon C(\rho,\tau)_{g^{-1}}\ni\sum\limits_{x\in \tau , y\in \rho}{_y}f_x \mapsto  \sum\limits_{x\in \tau , y\in \rho}{_{gy}}\af_g(f)_{gx} \in  C(\rho,\tau)_{g},$$
 for each $g\in {\rm mor}( \mathcal{G})$ is a partial action on $\bigcup\limits_{x\in \G_0}C(\rho,\tau)_x.$ Moreover, if $\cong_1$ is the equivalence relation determined by the orbit relation induced by $\af,$ then $\cong_1$ is a congruence in $\bigcup\limits_{x\in \G_0}C(\rho,\tau)_x.$ 
\end{prop}
\begin{proof} It is clear that $\af=(C(\rho,\tau)_g,\af_g)_{g\in {\rm mor}( \mathcal{G})}$ is a partial action of $\G$ on $\bigcup\limits_{x\in \G_0}C(\rho,\tau)_x.$ To check that $\cong_1$ is a congruence in $\bigcup\limits_{x\in \G_0}C(\rho,\tau)_x,$ consider morphisms $\f_1, \f_2, \f_3, \f_4$ in $\bigcup\limits_{x\in \G_0}C(\rho,\tau)_x,$ such that $\f_1\cong_1 \f_2$ in $C(\rho,\tau)$ and $\f_3\cong_1 \f_4$ in $C(\tau, \kappa).$ Then there are $n\in \Z^+,$ $g_1,\dots, g_n\in {\rm mor}( \mathcal{G})$ such that $\f_1\in C(\rho,\tau)_{{g_1}^{-1}}, \af_{g_1}(\f_1)\in C(\rho,\tau)_{g_2^{-1}}, \dots,  \af_{g_{n-1}}(\cdots (  \af_{g_2} ( \af_{g_1} (\f_1)))\cdots)\in C(\rho,\tau)_{{g_n}^{-1}}$ with 
\begin{equation}\label{afgn}\af_{g_n}(\af_{g_{n-1}}(\cdots (  \af_{g_2} ( \af_{g_1} (\f_1)))\cdots)=\f_2,\end{equation}
 and  $m\in \Z^+,$ $h_1,\dots, h_m\in {\rm mor}( \mathcal{G})$ such that $\f_3\in C(\tau, \kappa)_{{h_1}^{-1}},\af_{h_1}(\f_3)\in C(\tau, \kappa)_{h_2^{-1}}, \dots,  \af_{h_{m-1}}(\cdots (  \af_{h_2} ( \af_{h_1} (\f_3)))\cdots)\in C(\tau, \kappa)_{{h_m}^{-1}}$ and 
\begin{equation}\label{afhn}\af_{h_m}(\af_{h_{m-1}}(\cdots (  \af_{h_2} ( \af_{h_1} (\f_3)))\cdots))=\f_4.\end{equation}  

Now we use $\af_{(l_n,l_{n-1},\dots ,l_1)}(_u{r}_t)$ to denote  $
\af_{l_n}(\af_{l_{n-1}}(\cdots (  \af_{l_2} ( \af_{l_1} (_u{r}_t)))\cdots),$ where  $l_1,l_2,\dots ,l_n$ are morphisms in $\G,$ $u,t\in \C_0$ and $_u{r}_t\in {_u{\C}_t}$ such that $
\af_{l_n}(\af_{l_{n-1}}(\cdots (  \af_{l_2} ( \af_{l_1} (_u{r}_t)))\cdots),$ is well defined.
 
By (PA3) in Definition \ref{NYG} one may suppose w.l.o.g that $n$ in \eqref{afgn} coincides with  $m$ in \eqref{afhn}, then
\begin{align*}\f_4\f_2&=\af_{h_m}(\af_{h_{m-1}}(\cdots (  \af_{h_2} ( \af_{h_1} (\f_3)))\cdots))\af_{g_n}(\af_{g_{n-1}}(\cdots (  \af_{g_2} ( \af_{g_1} (\f_1)))\cdots)\\
&=\sum_{z\in \tau, w\in \kappa}\af_{(h_m,h_{m-1},\dots ,h_1)}(_w{f_3}_z)\sum_{x\in \rho, y\in \tau}\af_{(g_n,g_{n-1},\dots ,g_1)}(_y{f_1}_x)
\\
&\stackrel{\eqref{gfree}}=\sum_{y\in \tau, w\in \kappa, x\in \rho}\af_{(g_n,g_{n-1},\dots ,g_1)}(_w{f_3}_z)\af_{(g_n,g_{n-1},\dots ,g_1)}(_y{f_1}_x)
\\
&=\sum_{ w\in \kappa, x\in \rho}\af_{(g_n,g_{n-1},\cdots ,g_1)} (_w(f_3f_1)_x)\\
&=\af_{g_n}(\af_{g_{n-1}}(\dots (  \af_{g_2} ( \af_{g_1} (\f_3\f_1)))\cdots),
\end{align*}
from this we conclude that $\f_4\f_2\cong_1 \f_3\f_1$ as desired.
\end{proof}
\bigskip
Now the quotient semicategory is the semicategory $\C/\G$  whose set of objects are the $\cong$-equivalence classes   and  $_\tau\C/\G_\rho=\left(\bigcup\limits_{x\in \G_0}C(\rho,\tau)_x\right)/\cong_1,$ for all $\rho,\tau \in(\C/\G)_0.$ We say that  $\C$ is a Galois
covering of the quotient $\C/\G.$

%
\bigskip

In order to present the smash product category, we first recall that  $R$-algebra $S$  is said to be  graded by a groupoid $\G$ is 
 there is a set $\{ S_g \}_{g \in {\rm mor (\G)}}$ of $R$-submodules of
$S$ such that $S = \bigoplus\limits_{g \in  {\rm mor (\G)}} S_g$ and for all $g,h \in {\rm mor (\G)}$ 
$S_g S_h \subseteq S_{gh}$, if $(g,h) \in \G^2$ and $S_g S_h = \{ 0 \}$, otherwise. This and \cite[Section 3]{CM} motivates the following.

\begin{defn}
Let $\mathcal{G}$ be a groupoid, a $\mathcal{G}$-graded $R$-semicategory is an $R$-semicategory $\mathcal{B}$ 
together with a decomposition of each $R$-module morphism ${_y\mathcal{B}_x}=\bigoplus\limits_{g\in {\rm mor}(\mathcal{G})}{_y\mathcal{B}_x}^g$ such that, for each $x,y,z\in \B_0$  one has that
 ${_z\mathcal{B}_y}^t{_y\mathcal{B}_x}^s\subseteq {_z\B_x}^{ts}$ if $d(t)=r(s),$ and ${_z\mathcal{B}_y}^t{_y\mathcal{B}_x}^s=\{0\},$ otherwise.\end{defn}

Given a $\G$-graded semicategory $\B,$  we may construct the $\G$-graded semicategory $\B\otimes \G,$ where $(\B\otimes \G)_{0}=\B_0\times \G_0$ and $_{(y,f)}(\B\otimes \G)_{(x,e)}=\bigoplus_{g\in _f\mathcal{G}_e}{_y\mathcal{B}_x}^g,$ where $_f\mathcal{G}_e$ denotes the set of morphisms $g$ of $\G$ with $d(g)=e$ and $r(g)=f.$
\smallskip

Following \cite[Definition 3.1]{CM} we give the next.
\begin{defn}
Let $\mathcal{G}$ be a groupoid, $\mathcal{B}$ a $\mathcal{G}$-graded $R$-semicategory 
The smash product $R$-semicategory $\B \# \G$ has object set $\B_0\times {\rm mor}(\G),$ and if $(x,s)$ and $(y,t)$ are objects, then the $R$-module morphism $_{(y,t)}(\B \# \G)_{(x,s)}$ is defined as follows:
\begin{equation}\label{smash}_{(y,t)}(\B \# \G)_{(x,s)}=
  \begin{cases}
  {_y\B_x}^{t^{-1}s} & \text{if $r(t)=r(s)$},\\
  0, & \text{otherwise.}
  \end{cases}\end{equation}

In order to define the composition map
$$_{(z,u)}(\B \# \G)_{(y,t)}\otimes _{(y,t)}(\B \# \G)_{(x,s)}\to \,_{(z,u)}(\B \# \G)_{(x,s)} $$
note that if $(u^{-1},t), (t^{-1},s)\in \mathcal{G}^2,$ then $(u^{-1},s)\in \mathcal{G}^2$ and the left hand side is
$ {_z\B_y}^{u^{-1}t}\otimes  {_y\B_x}^{t^{-1}s}$
while the right hand side is $ {_z\B_x}^{u^{-1}s}$.
Then, the graded composition of $\B$ provides the
required map.\end{defn}
\begin{rem} Let $u,t,s\in \G$  besuch that $(u,t), (u,s)\in \mathcal{G}^2.$ Then $(t^{-1},s)$ and $ ((ut)^{-1},us)$ belong to $ \mathcal{G}^2,$ 
and 
  $(ut)^{-1}us= t^{-1}s$,  and  $_{(y,t)}(\B \# \G)_{(x,s)},$ From this one concludes that  $_{(y,ut)}(\B \# \G)_{(x,us)}$ are $R$-modules morphism from different objects which coincide as $R$-modules.
\end{rem}
The following result provide us sufficient conditions for which $\B\otimes\G$ and $\B \# \G$ to be  $R$-categories.
\begin{lem}\label{cat} Suppose that $\B$ is a $\G$-graded $R$-category and that $\G_0$ is finite, then   the $R$-semicategories $\B\otimes\G$ and $\B \# \G$ are $R$-categories.
\end{lem}
\begin{proof} Let $(x,e)\in (\B\otimes\G)_0$ and $(x,s)\in (\B \# \G)_{0},$  then $x\in \B_0$ and if $1_{_x\B_x}$ is the identity morphism of $x,$ then it   is the identity element of the $\G$-graded $R$-algebra  $_x\B_x.$ 
Let $1_{_x\B_x}=\sum_{u\in \G_0}{1_{_x\B_x}}^ u$  be the  homogeneous decomposition of  $1_{_x\B_x}$ then  the morphisms ${1_{_x\B_x}}^ {e}$ and ${1_{_x\B_x}}^ {d(s)}$ are  the identity morphism of $(x,e)$ and $(x,s),$ respectively.  Indeed, let  $(y,f)\in  (\B\otimes\G)_0$ and $h=\sum_{g\in _f\mathcal{G}_e}{_yh_x}^g
$ be a morphism in $_{(y,f)}(\B\otimes \G)_{(x,e)}$ then $h\in\, _{y}\B_{x}$ and $h=h1_{_x\B_x}.$ 
Now  for $u\in \G_0, u\neq e$ 
$$h{1_{_x\B_x}}^ u=\sum_{g\in _f\mathcal{G}_e}{_yh_x}^g{1_{_x\B_x}}^ u\in \sum_{g\in _f\mathcal{G}_e}{_y\mathcal{B}_x}^g{_x\mathcal{B}_x}^u=\{0\},$$ then 
$$h=h1_{_x\B_x}=\sum_{u\in \G_0}h{1_{_x\B_x}}^ u=h{1_{_x\B_x}}^ e,$$ as desired. In a similar way we show that ${1_{_x\B_x}}^ e\,l$ for all $l\in \,_{(x,e)}(\B\otimes \G)_{(y,f)}.$

Now we check that ${1_{_x\B_x}}^ {d(s)}$ is the identity morphism of $(x,s).$ Let $(y,t)\in  (\B \# \G)_{0},$  $g\in\, _{(y,t)}(\B \# \G)_{(x,s)}$ and suppose that $g\neq 0$ then $r(t)=r(s)$ and $g\in{_y\B_x}^{t^{-1}s}.$ We know that $g=g1_{_x\B_x},$  and if $e\in \G_0 $ is different from $d(s)$ we have that 
$g{1_{_x\B_x}}^ {e}\in {_y\B_x}^{t^{-1}s}{_x\B_x}^ {e}=\{0\},$ thus $g=g1_{_x\B_x}=g{1_{_x\B_x}}^ {d(s)},$ analogously one can show that $1_{_x\B_x}^ {d(s)}h=h,$ for all $h\in\, _{(x,s)}(\B \# \G)_{(y,t)}$.
\end{proof}

\begin{thm}\label{Galois}
Let $\B$ be a $\G$-graded category over $R$. Then  there is a global action $\alpha$ on $\B\#\G$ making $\B\#\G$ a free $\G$-category and a Galois covering of $\B\otimes \G.$
\end{thm}

\begin{proof}
 We define  a global  action $\alpha$ on $\B \# \G$ as follows. In  the set of objects $\B_0\times {\rm mor}(\G),$  take $g\in \G$ and define
 $$(\B_0\times {\rm mor}(\G))_g=\B_0\times \G(-, r(g))\,\,\text{ and}\,\, \alpha^0_g(x,s)=(x,gs),$$ 
for any $(x,s)\in (\B_0\times \G)_{g^{-1}}$. Then $\alpha^0=((\B_0\times {\rm mor}(\G))_g, \alpha^0_g)_{g\in {\rm mor}(\G)}$ is a global action of $\G$ on $\B_0\times \G.$  Now we define the action in the level of morphisms, for  $(x,s), (y,t)\in \B_0\times {\rm mor}(\G)$  and $g\in {\rm mor}(\G)$ let \begin{equation}\label{Ig}_{(y,t)}\I^g_{(x,s)}=
  \begin{cases}
  {_y\B_x}^{t^{-1}s} & \text{if $r(t)=r(s)= r(g)$ and},\\
  0, & \text{otherwise.}
  \end{cases}\end{equation} 
and we set 
$$\af^g(f)=f\in \,_{(y,gt')}\I^{g}_{(x,gs')},\,\, \text{ for all}\,\, f\in \,_{(y,t')}\I^{g^{-1}}_{(x,s')}.$$
 Then $\alpha=(\mathcal{I}^g, \alpha^g)_{g\in {\rm mor}\mathcal{(G)}}$ is a global action of $\G$ on $\B\#\G.$ Moreover, if $(x,s)\in (\B_0\times \G)_{g^{-1}}\cap (\B_0\times \G)_{h^{-1}}$ then
$g(x,s)=h(x,s)$ if and only if 
$g=h$ and we conclude that $\B\#\G$ a free $\G$-category. 

Now we check  that $\B\#\G$ is a Galois covering of $\B,$ that is $(\B\#\G)/\G=\B\otimes \G.$ It follows by Example \ref{exe3} that for any $(x,s)\in \B_0\times {\rm mor}(\G)$ its orbit is 
\begin{equation}\label{orbit}\G(x,s)=\{x\}\times \G(d(s),-)=\{x\}\times \G \cdot  s
\end{equation} where the orbit $\G\cdot  s$ is induced by the partial action given in Example \ref{exe3}.  Since the intersection of two different orbits is empty, the map $$(\B\#\G/\G)_0\ni \G(x,s)\mapsto (x,d(s))\in  \B_0\times \G_0$$ is a bijection. Now following the notation given in equation \eqref{cmor} and using \eqref{orbit}, for two orbits $\G(x,a), \G(y, b),$ where $a,b\in \G_0$ and $e\in \G_0,$ we write 
$$C(\G\cdot (x,a), \G\cdot (y, b))_e=\bigoplus\limits_{d(v)=b, d(u)=a}{_{(y,v)}\I^e_{(x,u)}}=\bigoplus\limits_{r(v)=r(u)= e, \atop d(v)=b, d(u)=a}{{_y\B_x}^{v^{-1}u}}.$$ Then  
\begin{align*}\bigcup_{e\in \G_0}C(\G\cdot (x,s), \G\cdot (y, t))_e&=\bigoplus\limits_{r(v)=r(u), \atop d(v)=b, d(u)=a}{{_y\B_x}^{v^{-1}u}}=\bigoplus\limits_{r(v)=r(u), \atop d(v)=b, d(u)=a}{_{(y,v)}(\B \# \G)_{(x,u)}}.
\end{align*}
Now by \eqref{smash} and \eqref{Ig} we get 
\[_{(y, e)}\I^e_{(x,s)}=
  \begin{cases}
  {_y\B_x}^{s} & \text{if $r(s)=e$ and},\\
  0, & \text{otherwise.}
  \end{cases}\,\,=\,\,\,_{(y,e)}(\B \# \G)_{(x,s)},\] 
for all $e\in \G_0.$
Thus  if $d(v)=r(u)$  we have that,
\begin{equation}\label{act}\af_v[\,_{(y,d(v))}(\B \# \G)_{(x,u)}]=\,_{(y,v)}(\B \# \G)_{(x,vu)}.\end{equation}Then 
\begin{align*}_{\G\cdot (y, b)}(\B\#\G/\G)_{\G\cdot (x,a)}&=\left[\bigoplus\limits_{r(v)=r(u), \atop d(v)=b, d(u)=a}{_{(y,v)}(\B \# \G)_{(x,u)}}\right]/\G
\\&
\stackrel{l=v^{-1}u}=\left[\bigoplus\limits_{ d(v)=b, d(l)=a \atop d(v)=r(l)}{_{(y,v)}(\B \# \G)_{(x,vl)}}\right]/\G\\&
\stackrel{\eqref{act}}\equiv\bigoplus\limits_{d(v)=b, d(l)=a \atop r(l)=d(v)}{_{(y,d(v))}(\B \# \G)_{(x,l)}}\\&
=\bigoplus\limits_{ l\in \G(a,b)}{_{(y,r(l))}(\B \# \G)_{(x,l)}}
\\&=\bigoplus\limits_{ l\in \G(a,b)}{_y\B _x}^l,
\\&=\,_{(y,b)}(\B\otimes \G)_{(x,a)}
\end{align*}
and we get that $\B\#\G$ is  a Galois covering of $\B\otimes \G.$
\end{proof}

The following results gives a relation between the partial skew groupoid category and the  smash product category.
\begin{prop} Let $\alpha$ be the global action of $\G$ on $\B\#\G$ defined in Theorem \ref{Galois}. Suppose that $\B$ is a $\G$-graded $R$-category and that $\G_0$ is finite. Then $(\B\#\G)\ast_{\alpha}\mathcal{G}$ is equivalent to $\B\otimes \G.$ In particular, $\B\#\G$ is a Galois covering of $(\B\#\G)\ast_{\alpha}\mathcal{G}.$
\end{prop}
\begin{proof} Since $\B$ is an $R$-category and  $\G_0$ is finite, then s  $\B\#\G$  and $\B\otimes \G$ are $R$-categories, thanks to Lemma \ref{cat}.  Moreover  $\alpha$ is  global, then by  Lemma \ref{skewcat}  we have that $(\B\#\G)\ast_{\alpha}\mathcal{G}$ is an $R$-category. Now, let $(x,t),(y,s)$ two objects of $\B\#\G,$ by definition we have
\begin{align*}_{(y,t)}[(\B\#\G)*_{\alpha}\G]_{(x,s)}&=\bigoplus\limits_{(g,s)\in \G^2}{_{(y,t)}{\I^g}_{(x,gs)}}
\\&\equiv \bigoplus\limits_{(g,s)\in \G^2, \atop r(t)=r(g)}{_y\B_x}^{t^{-1}gs}
\\&=\bigoplus\limits_{g\in _{r(t)}\G_{r(s)}}{_y\B_x}^{t^{-1}gs}\\&=\bigoplus\limits_{u\in_{d(t)}\G_{d(s)} }{_y\B_x}^{u}
\\&=\,_{(y,d(t))}(\B\otimes \G)_{(x,d(s))}.\end{align*} 
Thus there is a functor $F\colon (\B\#\G)\ast_{\alpha}\mathcal{G}\to \B\otimes \G$ such that $F(x,s)=(x,d(s))$ and $F$ is the identity on the morphism, then $F$ is clearly full and faithful,  and 
essentially surjective, then $F$ is an equivalence.
\end{proof}
We finish this work with an example to ilustrate our results.
\begin{exe}
We consider the  $R$-category $\B$ with
\begin{enumerate}
\item $\B_0=\{x, y\}.$
\item Given $ u,v \in \B_0$ let ${_u\B_v}=Re_1\oplus Re_2\oplus Re_3\oplus Re_4,$  where $e_1,e_2,e_3, e_4$ are pairwise orthogonal central idempotents with sum 1.
\item For all $u,v,w \in \B_0$ we have an $R$-bilinear map $\cdot: {_u\B_v}\times {_v\B_w}\to {_u\B_w};$ given by multiplication.
\end{enumerate}
Take the groupoid $\G=\{d(g), r(g), g, g^{-1}\}\,\, \text{with}\,\, \G_0=\{d(g), r(g)\}$.   Then  $\B$ is a $\G$-graded semicategory, where 
$${_u\mathcal{B}_v}^g=Re_1, \,\, {_u\mathcal{B}_v}^{g^{-1}}=Re_2,\,\, {_u\mathcal{B}_v}^{d(g)}=Re_3,\,\,{_u\mathcal{B}_v}^{r(g)}=Re_4,\,\,$$ for all $u,v\in \B_0,$
and the smash product $\B\#\G$ is defined thus where
\begin{enumerate}
\item[(i)] The objects:
\small
$$(\B\#\G)_0=\{(x,d(g)), (x,r(g)), (x,g), (x,g^{-1}), (y,d(g)), (y,r(g)), (y,g), (y,g^{-1})\}.$$
\normalsize
\item[(ii)] For all $u,v \in \B_0$ one has the $R$-module morphisms:

\begin{itemize}
\item $_{(u,d(g))}(\B \# \G)_{(v,d(g))}=\,_{(u,g)}(\B \# \G)_{(v,g)}= Re_3 $
\item $ _{(u,r(g))}(\B \# \G)_{(v,r(g))}=_{(u,g^{-1})}(\B \# \G)_{(v,g^{-1})}=Re_4.$
\item In the other possible cases
$_{(u,e)}(\B \# \G)_{(v,e')}=\{0\},\quad \text{where}\quad e,e' \in mor(\G).$
\end{itemize}

\end{enumerate}
Moreover, by Theorem \ref{Galois} there is a global action on $\B\#\G$  making it   a Galois covering  of  is $\B\otimes\G,$ where 
\begin{enumerate}
\item[(i)] The objects:
$$\B_0\otimes \G_0=\{(x,d(g)), (x,r(g)), (y,d(g)), (y,r(g))\}.$$
\item[(ii)] For all $u,v \in \B_0$ the $R$-module morphisms:
\begin{itemize}
\item 
$_{(u,d(g))}(\B \otimes \G)_{(v,d(g))}=_{(u,d(g))}(\B \otimes \G)_{(v,r(g))}=Re_3  $
 \item $ _{(u,r(g))}(\B \# \G)_{(v,r(g))}= _{(u,r(g))}(\B \# \G)_{(v,d(g))}=Re_4.$

\end{itemize}
\end{enumerate}
%
%

\end{exe}

\end{document}